\documentclass[11pt]{article} 
\usepackage{amsmath,amssymb,lscape,amsthm}
\usepackage[mathscr]{eucal}
\usepackage{amscd}
\newtheorem{theorem}{Theorem}[section]
\newtheorem{lemma}[theorem]{Lemma}
\newtheorem{proposition}[theorem]{Proposition}
\newtheorem{corollary}[theorem]{Corollary}
\theoremstyle{definition}
\newtheorem{definition}[theorem]{Definition}
\theoremstyle{remark}
\newtheorem{remark}[theorem]{Remark}
\numberwithin{equation}{section}
\title{Biharmonic orbits of isotropy representations of symmetric spaces} 
\author{Shinji Ohno} 
\date{\today} 

\begin{document}
\maketitle
\begin{abstract}
In this paper, we give a necessarly and sufficient condition for orbits of linear isotropy representations of Riemannian symmetric spaces are biharmonic submanifolds in hyperspheres in Euclidean spaces.
In particular, we obtain examples of biharmonic submanifolds in hyperspheres whose co-dimension is greater than one. 
\end{abstract}
\section{Introduction}\label{Introduction}

J. Eells and L. Lemaire(\cite{EL1}) introduced the notion of biharmonic map as a generalization of the notion of harmonic map. 
For a smooth map $\varphi$ from a Riemannian manifold $(M, g)$ into an another Riemannian manifold $(N, h)$,
$\varphi$ is said to be harmonic if it is a critical point of the energy functional defined by 
\begin{align*}
E(\varphi)=\frac{1}{2}\int_{M} |d\varphi|^{2}d\mu_{g}.
\end{align*} 
The Euler-Lagrange equation is given by the vanishing of the tension field $\tau_{\varphi}$. 
Harmonic maps are studied by many mathematicians.
The biharmonic maps, which is a generalization of the harmonic map, 
is defined as a critical point of bienergy functional
\begin{align*}
E_{2}(\varphi)=\frac{1}{2}\int_{M} \| \tau_{\varphi}\|^{2}d\mu_{g}.
\end{align*}
Similar to harmonic maps, 
Biharmonic maps are characterized by 
the Euler-Lagrange equation $\tau_{2, \varphi}=0$ 
where $\tau_{2, \varphi}$ is the bitension field of $\varphi$.
It is known that the equation $\tau_{2, \varphi}=0$ is a fourth order partial differential equation. 
By definition, 
harmonic maps are biharmonic maps. 
On the other hand, 
a biharmonic map is not necessary harmonic.
The B. Y. Chen's conjecture is to ask whether every biharmonic submanifold of the Euclidean space $\mathbb{R}^{n}$ must be harmonic, i.e., minimal(\cite{C}).
It was partially solved positively.
For example, 
Akutagawa and Maeta showed (\cite{AM}) that every 
complete properly immersed biharmonic submanifold in the Euclidean space 
${\mathbb R}^n$ must be minimal.
Furthermore, 
it is known (cf. \cite{NU1}, \cite{NU2}, \cite{NUG})
that every biharmonic map of a complete Riemannian manifold
into an another Riemannian manifold of non-positive sectional curvature
with finite energy and finite bienergy must be harmonic.  

On the contrary, 
for the target Riemannian manifold $(N, h)$ of non-negative sectional curvature, 
there exist examples of biharmonic submanifolds which are not harmonic. 
A biharmonic submanifold is called proper if it is not harmonic. 
T. Ichiyama, J. Inoguchi and  H. Urakawa (\cite{IIU}) classified 
homogeneous hypersurfaces which are proper biharmonic in the hypersphere in Euclidean spaces.
More generally, biharmonic homogeneous hypersurfaces in compact symmetric spaces are studied in 
\cite{OSU} and \cite{IS}.
Furthermore, S. Ohno, T. Sakai and H. Urakawa construct higher co-dimensional biharmonic submanifolds in compact symmetric spaces as orbits of Hermann actions which are a generalization of isotropy actions of compact symmetric spaces(\cite{OSU2}).
However, since the rank of hyperspheres are one, the cohomogeneity of Hermann actions on hyperspheres are one.
Therefore, in orbits of Hermann actions on hyperspheres, there is no proper biharmonic submanifolds whose co-dimension is greater than one. 
A. Blamu\c{s}, S. Montaldo and C. Oniciuc give new examples of proper biharmonic submanifolds in spheres and classification of biharmonic submanifolds which are the direct products of some spheres in the unit sphere in \cite{BMO1} and \cite{BMO2}. 

In this paper, using root systems,
we describe a necessary and sufficient condition for an orbit of the linear isotropy representation of a Riemannian symmetric space to be biharmonic in hypersphere, 
and give examples of proper biharmonic submanifolds in hypersphere whose co-dimension is greater than one.

The organization of this paper is as follows.
In Section~\ref{Preliminaries}, we prepare the foundation for the following sections.
In \ref{Biharmonic isometric immersions}, we describe biharmonic isometric immersions.
In particular, we explain that
for an isometric immersion whose tension field is parallel, the biharmonic property is characterized by a condition of the second fundamental form of the isometric immersion (Theorem~\ref{parallel}).
In \ref{Compact symmetric pair and hyperpolar representation}, 
we examine the linear isotropy representations of Riemannian symmetric spaces.
We state that the second fundamental form of an orbit of the linear isotropy representation of a Riemanniam symmetric space is described by the root system of the Riemannian symmetric space.
Moreover, we show the tension field of an orbit of the linear isotropy representation of a Riemanniam symmetric space is parallel with respect to the normal connection.
In Section~\ref{Main theorem and Examples}, 
we show Theorem~\ref{characterizetion:root} and give new examples of proper biharmonic submanifolds of hyper spheres.

\section{Preliminaries} \label{Preliminaries}

\subsection{Biharmonic isometric immersions}\label{Biharmonic isometric immersions}
In this section, 
we describe biharmonic isometric immersions. 
Let $(M, g)$ and $(N, h)$ be Riemannian manifolds,
and  
$\varphi$ be a smooth map from $M$ into $N$. 
We denote by $\nabla, \nabla^{h}$ the Levi-Civita connections on $TM, TN $ of $(M, g), (N, h)$, 
and by $\overline{\nabla}, \widetilde{\nabla}$ the induced connections on $\varphi^{-1}TN, T^{\ast}M \otimes \varphi^{-1}TN$ respectively.
Let $B_{\varphi}$ denotes the second fundamental form of $\varphi$, i.e.
\begin{align*}
B_{\varphi}(X, Y)=\overline{\nabla}_{X}(d\varphi(Y))-d\varphi(\nabla_{X}Y)
\end{align*}
for 
$X, Y \in \mathfrak{X}(M)$. 
Then we define the tension field $\tau_{\varphi}$ of $\varphi$ by  
\begin{align*}
\tau_{\varphi}=\sum_{i=1}^{m}B_{\varphi}(e_{i}, e_{i})
\end{align*}
where $\{e_{i}\}_{i=1}^{m}$ is a locally defined orthonormal frame field on $M$.
\begin{definition}\label{harm}
A smooth map $\varphi$ is called harmonic if $\tau_{\varphi}=0$.
If a harmonic map $\varphi$ is an isometric embedding, then the image $\varphi(M) \subset N$ is called a harmonic submanifold.
\end{definition}
\begin{remark}\rm
When a smooth map $\varphi$ is an isometric immersion, 
$\tau_{\varphi}$ coincides with the mean curvature vector field of $\varphi$. 
Then,  
a harmonic map is a minimal immersion, 
and 
a harmonic submanifold is a minimal submanifold.
\end{remark}

To define the notion of biharmonic maps, 
we define the Jacobi operator $J$.  
For $V\in \Gamma(\varphi^{-1}TN)$ 
\begin{align*}
J(V):=\overline{\Delta}V-\mathcal{R}(V),
\end{align*}
where 
$\overline{\Delta}V=\overline{\nabla}^{\ast}\overline{\nabla}V=\sum_{i=1}^{m}\{\overline{\nabla}_{e_{i}}\overline{\nabla}_{e_{i}}V-\overline{\nabla}_{\nabla_{e_{i}}e_{i}}V \}$, 
$\mathcal{R}(V)=\sum_{i=1}^{m}R^{h}(V, d\varphi(e_{i}))d\varphi(e_{i})$.
Here $R^{h}$ is the curvature tensor field of $N$. 
Then we set 
\begin{align*}
\tau_{2, \varphi}=J(\tau_{\varphi}).
\end{align*}
The vector field $\tau_{2, \varphi}$ is called a bitention field of $\varphi$. 
\begin{definition}\label{biharm}
A smooth map $\varphi$ is called biharmonic if $\tau_{2, \varphi}=0$.
If a biharmonic map $\varphi$ is an isometric embedding, then the image $\varphi(M) \subset N$ is called by  biharmonic submanifold.
\end{definition}
Then we have the following theorem. 
\begin{theorem}[\cite{OSU}]\label{parallel}
Let $\varphi :M\to N$ be a isometric immersion which satisfies that $\nabla^{\perp}_{X}\tau_{\varphi}=0$ for all $X \in \mathfrak{X}(M)$.
Then $\varphi$ is biharmonic if and only if  
\begin{align}\label{biheqn}
\sum_{i=1}^{m}R^{h}(\tau_{\varphi}, d\varphi(e_{i}))d\varphi(e_{i})
=\sum_{i, j=1}^{m}h(\tau_{\varphi}, B_{\varphi}(e_{i}, e_{j}))B_{\varphi}(e_{i}, e_{j}).
\end{align}  
\end{theorem}
\begin{remark}\rm
The condition (\ref{biheqn}) is equivalent to the following equation, 
\begin{align}\label{biheqn2}
\sum_{i=1}^{m}R^{h}(\tau_{\varphi}, d\varphi(e_{i}))d\varphi(e_{i})
=\sum_{i=1}^{m}B_{\varphi}(A_{\tau_{\varphi}}e_{i}, e_{i})).
\end{align} 
Here $A_{\tau_{\varphi}}$ is the shape operator of $\varphi$ with respect to $\tau_{\varphi}$. 
It holds
$
g(A_{\tau_{\varphi}}X, Y)=h(B_{\varphi}(X,Y), \tau_{\varphi}).
$ 
\end{remark}
\subsection{Compact symmetric pair and the second fundamental form of R-spaces in spheres} \label{Compact symmetric pair and hyperpolar representation}
In this section, 
we express the second fundamental form of orbits of the linear isotropy representations of Riemannian symmetric spaces in hyperspheres in terms of root systems.
Let $G$ be a compact connected semisimple Lie group and $\sigma$ be an involutive automorphism of $G$. 
We take a subgroup $K$ of $G$ which satisfies $\mathrm{Fix}(\sigma, G)_{0}\subset K \subset \mathrm{Fix}(\sigma, G)$,
where $\mathrm{Fix}(\sigma, G)$ is the subgroup of the fixed point set of $\sigma$ and $\mathrm{Fix}(\sigma, G)_{0}$ is the identity component of $\mathrm{Fix}(\sigma, G)$.
Let $\mathfrak{g}$ and $\mathfrak{k}$ denote the Lie algebras of $G$ and $K$ respectively. 
The involutive automorphism of $\mathfrak{g}$ induced from $\sigma$ will be also denoted by $\sigma$.
Then, by definition of $K$,
we have $\mathfrak{k}=\{X\in \mathfrak{g} \mid \sigma(X)=X\}$.
Take an $\mathrm{Ad}(K)$-invariant inner product $\langle \cdot, \cdot \rangle$ on $\mathfrak{g}$.
Then 
\begin{align*}
\mathfrak{g}=\mathfrak{k}\oplus \mathfrak{m}
\end{align*}
is an orthogonal direct sum decomposition of $\mathfrak{g}$ where $\mathfrak{m}=\{ X\in \mathfrak{g} \mid \sigma(X)=-X\}$.
Let $\pi$ denotes the natural projection from $G$ into the coset manifold $G/K$.
The tangent space $T_{\pi(e)}G/K$ of $G/K$ at the origin $\pi(e)$ is identified
with $\mathfrak{m}$ in a natural way,
where $e$ is the identity element of $G$.
Then the inner product $\langle \cdot, \cdot \rangle$ induces a $G$-invariant Riemannian metric on $G/K$.
We denote the Riemannian metric on $G/K$ by the same symbol $\langle \cdot, \cdot \rangle$.
Then $G/K$ is a compact Riemannian symmetric space with respect to $\langle \cdot, \cdot \rangle$.
The group $G$ acts on $G/K$ isometrically by $L_{y}(xK):=yxK \ (x, y \in G)$ .
Thus the subgroup $K$ acts on $G/K$ isometrically, and the action is called by the isotropy action of $G/K$. 
Since for any $k \in K$, $L_{k}$ fixes $o:=eK \in G/K$, the differential $dL_{k}$ of $L_{k}$ at $o$ gives a linear trans formation on $T_{o}G/K$. 
For each $k, k' \in K$, $L_{k}\circ L_{k'}=L_{kk'}$ holds. 
Thus, $K$ has a representation on $T_{o}G/K$, 
and this representation on $T_{o}G/K$ is called the linear isotropy representation of $G/K$. 
On the other hand, 
the differential $\mathrm{Ad}(x)$ of an inner automorphism $\mathrm{I}_{x}$ at $e$ is an automorphism on $\mathfrak{g}$ for $x \in G$, 
where $\mathrm{I}_{x}(y)=xyx^{-1} \ (y \in G)$.
Then we have 
\begin{align}
\mathrm{Ad}(k)\mathfrak{k}=\mathfrak{k},\ 
\mathrm{Ad}(k)\mathfrak{m}=\mathfrak{m} 
\end{align}
for $k \in K$.
Therefore, 
$K$ has a representation on $\mathfrak{m}$.
Then we have 
\begin{align*}
(d\pi)_{e}(\mathrm{Ad}(k)X)=(dL_{k})_{o}((d\pi)_{e}(X))\quad (k\in K, X\in \mathfrak{m}).
\end{align*}
Hence, 
hereafter we consider the representation $K$ on $\mathfrak{m}$.
Take and fix a maximal abelian subspace $\mathfrak{a}$ of $\mathfrak{m}$.
Then it is known 
\begin{align*}
\mathrm{Ad}(K)\mathfrak{a}=\mathfrak{m}.
\end{align*}
Since 
$\langle \mathrm{Ad}(k)X, \mathrm{Ad}(k)Y \rangle =\langle X, Y \rangle\quad (X, Y \in \mathfrak{m})$ holds for all $k \in K$, 
$\mathrm{Ad}(k)$ preserving the unit sphere $S$ in $\mathfrak{m}$. 
For each 
$H \in S$, 
the orbit $\mathrm{Ad}(K)H$ in $S$ is a submanifold of $S$,  
and $\mathrm{Ad}(K)H$ is called a R-space. 
In particular, if $\mathrm{Ad}(K)H$ in $S$ is a minimal submanifold, then it is called a minimal R-space. 
We would like to examine a necessary and sufficient condition for a R-space $\mathrm{Ad}(K)H$ in $S$ is a biharmonic submanifold. 
In order to apply Theorem~\ref{parallel} to $\mathrm{Ad}(K)H$ in $S$,  
we calculate the second fundamental form of $\mathrm{Ad}(K)H$ in $S$, using the root system of $G/K$. 
We define subspaces of $\mathfrak{g}$ as follows:
\begin{align*}
\mathfrak{k}_{0}=\{X\in \mathfrak{k}\mid [H', X]=0 \quad (H' \in \mathfrak{a})\},
\end{align*}
for $\lambda \in\mathfrak{a}\setminus \{0\}$, 
\begin{align*}
\mathfrak{k}_{\lambda }&=\{X \in \mathfrak{k} \mid [H',[H', X]]=-\langle \lambda , H'\rangle^{2}X \quad (H'\in \mathfrak{a}) \}, \\
\mathfrak{m}_{\lambda }&=\{X \in \mathfrak{m} \mid [H',[H', X]]=-\langle \lambda , H'\rangle^{2}X \quad (H'\in \mathfrak{a}) \}.
\end{align*}
We set 
$\Sigma=\{\lambda \in \mathfrak{a}\setminus \{0\} \mid \mathfrak{k}_{\lambda }\neq \{0\}\}$ and 
$m(\lambda )=\dim \mathfrak{k}_{\lambda }$.
The subset $\Sigma$ in $\mathfrak{a}$ is called the root system of $G/K$.
Since
$\mathfrak{k}_{\lambda }=\mathfrak{k}_{-\lambda }$, 
if $\lambda \in \Sigma$, then $-\lambda \in \Sigma $.
Fix a basis of $\mathfrak{a}$ and define a lexicographic ordering $>$ on $\mathfrak{a}$ with respect to the basis of $\mathfrak{a}$, and set  
$
\Sigma^{+}=\{\lambda \in \Sigma \mid \lambda >0\}.
$ 
Let $\Pi=\{\alpha_{1}, \ldots \alpha_{r}\}$ be a set of simple roots of $\Sigma$ 
where $r= \dim \mathfrak{a}$.
For 
$1\leq i\leq r$, 
we define $H_{\alpha_i}\in \mathfrak{a}$ by 
\begin{align*}
\langle H_{\alpha_i}, \alpha_{j}\rangle=\delta_{i,j}\quad (1\leq j\leq r).
\end{align*}
Here, 
$\delta_{i,j}$ is the Kronecker delta. 
Since 
$\Pi$ is a basis of $\mathfrak{a}$, 
$\{H_{\alpha_1}, \ldots , H_{r}\}$ is also a basis of $\mathfrak{a}$. 
Using $\Pi$ and $\{H_{\alpha_1}, \ldots , H_{\alpha_r}\}$, we set an open subset $\mathcal{C}$ of $\mathfrak{a}$ as follows:
\begin{align*}
\mathcal{C}=\{H \in \mathfrak{a} \mid \langle \alpha , H\rangle >0 \quad (\alpha \in \Pi)\}
=\left\{ \left. \sum_{i=1}^{r} x_{i}H_{\alpha_i} \ \right| \ x_{i}>0 \right\}.
\end{align*}
The closure $\overline{\mathcal{C}}$ of $\mathcal{C}$ is given as  
\begin{align*}
\overline{\mathcal{C}}=\{H \in \mathfrak{a} \mid \langle \alpha , H\rangle \geq 0 \quad (\alpha \in \Pi)\}
=\left\{ \left. \sum_{i=1}^{r} x_{i}H_{\alpha_i} \ \right| \ x_{i}\geq 0 \right\}.
\end{align*}
Then,  
\begin{align}
\mathrm{Ad}(K)\overline{\mathcal{C}}=\mathfrak{m}
\end{align}
holds. 
For each subset $\Delta \subset \Pi$, we set 
\begin{align*}
\mathcal{C}^{\Delta}=\{H \in \mathfrak{a} \mid \langle \alpha, H\rangle >0, \langle \beta, H\rangle =0 \quad (\alpha \in \Pi,\ \beta \in \Pi \setminus \Delta )\}.
\end{align*}
Then we have the cell decomposition of $\overline{\mathcal{C}}$
\begin{align}\label{celldecomp}
\overline{\mathcal{C}}=\bigcup_{\Delta \subset \Pi}\mathcal{C}^{\Delta} \quad (\text{disjoint union}).
\end{align}
 
The set $\overline{\mathcal{C}}$ is the orbit space of the representation of $\mathrm{Ad}(K)$ on $\mathfrak{m}$.
Therefore, it is sufficient to consider $\mathrm{Ad}(K)H$ in $S$ for $H \in \overline{\mathcal{C}} \cap S$. 

Hereafter, we assume $H \in \overline{\mathcal{C}}\cap S$. 
In order to compute the second fundamental form of $\mathrm{Ad}(K)H$ in $S$, 
we use the following lemma. 
\begin{lemma}\label{onb}
For each $\lambda \in \Sigma^{+}$, 
there exist orthonormal bases $\{S_{\lambda, i}\}_{i=1}^{m(\lambda )}$ and $\{T_{\lambda, i}\}_{i=1}^{m(\lambda )}$ of $\mathfrak{k}_{\lambda }$ and $\mathfrak{m}_{\lambda}$ respectively
such that for any $H'\in \mathfrak{a}$, 
\begin{align*}
[H', S_{\lambda, i}]&=\langle \lambda ,  H'\rangle T_{\lambda, i},\quad 
[H', T_{\lambda, i}]=-\langle \lambda ,  H'\rangle S_{\lambda, i},\quad 
[S_{\lambda, i}, T_{\lambda, i}]=\lambda, \\
\mathrm{Ad}(\exp(H'))S_{\lambda , i}&=\cos \langle \lambda ,  H'\rangle S_{\lambda , i}+\sin\langle \lambda ,  H'\rangle T_{\lambda,i}\\
\mathrm{Ad}(\exp(H'))T_{\lambda , i}&=-\sin \langle \lambda ,  H'\rangle S_{\lambda , i}+\cos\langle \lambda ,  H'\rangle T_{\lambda,i}
\end{align*}
holds. 
\end{lemma}
By Lemma~\ref{onb}, we have the following direct sum decompositions:  
\begin{align*}
\mathfrak{k}=\mathfrak{k}_{0}\oplus \sum_{\lambda \in \Sigma^{+}}\mathfrak{k}_{\lambda}=
\mathfrak{k}_{0}\oplus \sum_{\lambda \in \Sigma^{+}}\sum_{i=1}^{m(\lambda)} \mathbb{R}\cdot S_{\lambda , i}, \\
\mathfrak{m}=\mathfrak{a}\oplus \sum_{\lambda \in \Sigma^{+}}\mathfrak{m}_{\lambda}=
\mathfrak{a}\oplus \sum_{\lambda \in \Sigma^{+}}\sum_{i=1}^{m(\lambda)} \mathbb{R}\cdot T_{\lambda , i}.
\end{align*}
The tangent space $T_{H}(\mathrm{Ad}(K)H)$ and the normal space $T^{\perp}_{H}(\mathrm{Ad}(K)H)$ in $S$ of 
$\mathrm{Ad}(K)H$ at the point $H \in \overline{\mathcal{C}} \cap S$ is given as 
\begin{align*}
T_{H}(\mathrm{Ad}(K)H)&=
\left\{ \left. \left. \frac{d}{dt} \mathrm{Ad}(\exp(tX))H \right|_{t=0}\quad  \right| \quad  X \in \mathfrak{k}\right\}
=\{[X, H] \mid X\in \mathfrak{k}\}=[\mathfrak{k}, H]\\
&=\sum_{\lambda \in \Sigma^{+}}\sum_{i=1}^{m(\lambda)} \mathbb{R}\cdot (\langle \lambda , H \rangle T_{\lambda , i})
=\sum_{\lambda \in \Sigma^{+}, \langle \lambda , H\rangle \neq0}\sum_{i=1}^{m(\lambda)} \mathbb{R}\cdot T_{\lambda , i} \\
&=\sum_{\lambda \in \Sigma^{+}, \langle \lambda , H\rangle \neq0}\mathfrak{m}_{\lambda}, \\
T^{\perp}_{H}(\mathrm{Ad}(K)H)&=
\left(
\mathfrak{a}\oplus \sum_{\lambda \in \Sigma^{+}, \langle \lambda , H\rangle =0}\mathfrak{m}_{\lambda}
\right)
\cap T_{H}S.
\end{align*}

For $H \in \mathfrak{a}\cap S$,
we set $\Sigma_{H}=\{\lambda \in \Sigma \mid \langle \lambda, H\rangle =0\}$. 
Let $X^{T}$ denotes the tangent vector in $T_{H}S=\{Y \in \mathfrak{m}\mid \langle Y, H\rangle =0\}$ which is defined as  
\begin{align*}
X^{T}=X-\langle X, H\rangle H 
\end{align*} 
for $X\in \mathfrak{m}$.
The vector $X^{T}$ depends on $H\in \mathfrak{a}\cap S$. 

Then we compute the covariant derivative of the orbit $\mathrm{Ad}(K)H$ in $S$. 
Let $\nabla^{S}$ and $\nabla$ denote the Levi-Civita connections of $S$ and $\mathrm{Ad}(K)H$, respectively. 
For each $\lambda \in \Sigma^{+} \setminus \Sigma_{H},\ 1\leq i\leq m(\lambda)$, we define a vector field $(T_{\lambda, i})^{\ast}$ on $\mathfrak{m}$ by 
 \begin{align*}
(T_{\lambda, i})^{\ast}_{X}&=
\left.\frac{d}{dt} \mathrm{Ad}\left(\exp \left( -\frac{tS_{\lambda, i}}{\langle \lambda , H\rangle}\right) \right)X \right|_{t=0}\\
&=-\frac{[S_{\lambda, i}, X]}{\langle\lambda , H\rangle}
\end{align*}
for $X \in \mathfrak{m}$.
Then $(T_{\lambda, i})^{\ast}_{H}=T_{\lambda, i}$ holds. 
Moreover, for each $X \in \mathrm{Ad}(K)H$,  $(T_{\lambda, i})^{\ast}_{X}$ is a tangent vector of $\mathrm{Ad}(K)H$ at $X$. 
Then we have the following proposition. 
\begin{proposition}\label{covdiff}
For each $\lambda, \mu \in \Sigma^{+}\setminus \Sigma_{H},\ 1\leq i\leq m(\lambda ), 1\leq j\leq m(\mu)$,
we have  
\begin{align*}
\left(\nabla^{S}_{(T_{\lambda, i})^{\ast}}(T_{\mu, j})^{\ast}\right)_{H}
=
-\frac{1}{\langle \mu, H\rangle }
\left(
[S_{\mu, j}, T_{\lambda, i}]
\right)^{T}.
\end{align*}
\end{proposition}
\begin{proof}
For $\lambda, \mu \in \Sigma^{+}\setminus \Sigma_{H},\ 1\leq i\leq m(\lambda ), 1\leq j\leq m(\mu)$,
we set a smooth curve
\begin{align}\label{ct}
c(t)=\mathrm{Ad}\left(\exp \left( -\frac{tS_{\lambda, i}}{\langle \lambda , H\rangle}\right) \right)H
\end{align}
 in $\mathrm{Ad}(K)H$. 
Then 
\begin{align*}
\left(\nabla^{S}_{(T_{\lambda, i})^{\ast}}(T_{\mu, j})^{\ast}\right)_{H}
=\left(
\left.\frac{d}{dt}(T_{\mu, j})^{\ast}_{c(t)}\right|_{t=0}
\right)^{T}
\end{align*}
holds.  
Thus we have 
\begin{align*}
\left.\frac{d}{dt}(T_{\mu, j})^{\ast}_{c(t)}\right|_{t=0}
=&\left.\frac{d}{dt} \frac{-1}{\langle \mu, H\rangle}[S_{\lambda, i}, c(t)]\right|_{t=0}\\
=&-\frac{1}{\langle \mu, H\rangle}\left[ S_{\lambda, i},  -\frac{1}{\langle \lambda, H\rangle}[S_{\lambda, i},H]\right] \\
=&-\frac{1}{\langle \mu, H\rangle}[S_{\mu, j}, T_{\lambda, i}].
\end{align*}
Therefor we obtain
\begin{align*}
\left(\nabla^{S}_{(T_{\lambda, i})^{\ast}}(T_{\mu, j})^{\ast}\right)_{H}
=
-\frac{1}{\langle \mu, H\rangle }
\left(
[S_{\mu, j}, T_{\lambda, i}]
\right)^{T}.
\end{align*} 
\end{proof}
By using Proposition \ref{covdiff}, we can express the tension field $\tau_{H}$ of $\mathrm{Ad}(K)H$ in $S$.
\begin{corollary}[\cite{KO}]
Let $\tau_{H}$ be the tension field of $\mathrm{Ad}(K)H$ in $S$. Then, 
\begin{align}
(\tau_{H})_{H}=-\left( \sum_{\lambda \in \Sigma^{+}\setminus \Sigma_{H}} \frac{m(\lambda)}{\langle \lambda, H\rangle} \lambda \right)^{T}
\end{align}
holds. 
In particular, $(\tau_{H})_{H}\in \mathfrak{a}$ holds. 
\end{corollary}
We set, 
\begin{align}
\widetilde{\tau_{H}}=-\sum_{\lambda \in \Sigma^{+}\setminus \Sigma_{H}} \frac{m(\lambda)}{\langle \lambda, H\rangle} \lambda 
\end{align}
then $\widetilde{\tau_{H}}^{T}=(\tau_{H})_{H}$ and  
\begin{align*}
\left\langle \widetilde{\tau_{H}}, H\right\rangle
=-\sum_{\lambda \in \Sigma^{+}\setminus \Sigma_{H}} \frac{m(\lambda)}{\langle \lambda, H\rangle} \langle \lambda, H\rangle 
=-\sum_{\lambda \in \Sigma^{+}\setminus \Sigma_{H}} m(\lambda)
=-\dim (\mathrm{Ad}(K)H)
\end{align*}
holds. 
Therefore, we have 
\begin{align}
\widetilde{\tau_{H}}=(\tau_{H})_{H} -\dim (\mathrm{Ad}(K)H)H.
\end{align}
Thus,  
$(\tau_{H})_{H}=0$ if and only if 
$\widetilde{\tau_{H}}=-\dim (\mathrm{Ad}(K)H)H$. 
In order to apply Theorem~\ref{biharm} to $\mathrm{Ad}(K)H$ in $S$, 
we need the following lemma.
\begin{lemma}\label{parallelmeancurvature}
For any $X\in T_{H}\mathrm{Ad}(K)H$, 
\begin{align}
\left(
\nabla^{\perp}_{X}\tau_{H}
\right)_{H}=0
\end{align}
holds.
\end{lemma}
\begin{proof}
Since for each $X\in T_{H}\mathrm{Ad}(K)H$, $X$ is expressed by linear combination of $\{T_{\lambda , i} \mid \lambda \in \Sigma \setminus \Sigma_{H}, 1\leq i\leq m(\lambda )\}$, 
it is sufficient to prove that 
\begin{align*}
\left(
\nabla^{\perp}_{T_{\lambda, i}}\tau_{H}
\right)_{H}=0
\end{align*}
for $\lambda \in \Sigma \setminus \Sigma_{H}, 1\leq i\leq m(\lambda )$.
Since $\tau_{H}$ is invariant under an isometry, 
we have 
\begin{align*}
(\tau_{H})_{\mathrm{Ad}(k)H}=\mathrm{Ad}(k)(\tau_{H})_{H}
\end{align*}
for each $k \in K$. 
Let $X^{\perp}$ denotes the normal part of $X\in T_{H}S$ in $T^{\perp}_{H}\mathrm{Ad}(K)H$. 
Then, by using the smooth curve $c(t)$ defined in (\ref{ct}), 
\begin{align*}
&\left(
\nabla^{\perp}_{T_{\lambda, i}}\tau_{H}
\right)_{H}
=\left(
\nabla^{S}_{T_{\lambda, i}}\tau_{H}
\right)_{H}^{\perp}\\
&=
\left(\left. \frac{d}{dt}(\tau_{H})_{c(t)}\right|_{t=0} \right)^{\perp}
=
\left(\left. \frac{d}{dt}\mathrm{Ad}\left( \exp \left( -\frac{tS_{\lambda, i}}{\langle \lambda , H\rangle }\right)\right) (\tau_{H})_{H}\right|_{t=0} \right)^{\perp}\\
&=\left(-\frac{1}{\langle \lambda , H\rangle }[S_{\lambda. i}, (\tau_{H})_{H}] \right)^{\perp}\\
&=\left(\frac{\langle \lambda , (\tau_{H})_{H}\rangle }{\langle \lambda , H\rangle }T_{\lambda, i} \right)^{\perp}=0.
\end{align*}
\end{proof}
\section{Main theorem and Examples} \label{Main theorem and Examples}
By Lemma~\ref{parallelmeancurvature}, we can apply Theorem~\ref{parallel} to the orbit $\mathrm{Ad}(K)H$ in $S$.
Then we have the following theorem. 
\begin{theorem}\label{characterizetion:root}
Let $H \in \mathfrak{a}\cap S$. 
Then, 
$\mathrm{Ad}(K)H$ is biharmonic in $S$ if and only of  
\begin{align}\label{biheqn3}
\dim (\mathrm{Ad}(K)H)(\tau_{H})_{H}=\sum_{\lambda \in \Sigma^{+}\setminus \Sigma_{H}}m(\lambda )\frac{\langle \lambda , (\tau_{H})_{H}\rangle}{\langle \lambda , H\rangle^{2}} (\lambda)^{T}.
\end{align} 
\end{theorem}
\begin{proof}
We compute both sides of the equation (\ref{biheqn2}).

For each $\lambda \in \Sigma^{+}\setminus \Sigma_{H},\ 1\leq i\leq m(\lambda)$, 
we have 
\begin{align*}
R((\tau_{H})_{H}, T_{\lambda , i})T_{\lambda, i}=\frac{1}{\langle H, H\rangle }(\tau_{H})_{H},
\end{align*}
where $R$ is the curvature tensor of $S$. 
Since $\langle H, H\rangle=1$, 
\begin{align*}
&\sum_{\lambda \in \Sigma^{+}\setminus \Sigma_{H}} \sum_{i=1}^{m(\lambda)} R((\tau_{H})_{H}, T_{\lambda , i})T_{\lambda, i}
=\sum_{\lambda, \in \Sigma^{+}\setminus \Sigma_{H}}m(\lambda)(\tau_{H})_{H}\\
&=\dim (\mathrm{Ad}(K)H)(\tau_{H})_{H}
\end{align*}
holds.
Let $B(\cdot , \cdot )$ denotes the second fundamental form of $\mathrm{Ad}(K)H$ in $S$.
For $\lambda,  \mu \in \Sigma^{+}\setminus \Sigma_{H},\ 1\leq i\leq m(\lambda), 1\leq j\leq m(\mu)$,
\begin{align*}
&\langle A_{(\tau_{H})_{H}}T_{\lambda ,i}, T_{\mu, j}\rangle 
= \left\langle (\tau_{H})_{H}, B(T_{\lambda ,i},T_{\mu, j}) \right\rangle 
= \left\langle (\tau_{H})_{H}, \left( \nabla^{S}_{(T_{\lambda ,i})^{\ast}}(T_{\mu, j})^{\ast}\right)_{H}^{T} \right\rangle \\
&=\left\langle (\tau_{H})_{H}, -\frac{1}{\langle \mu , H\rangle }[S_{\mu , j}, T_{\lambda, i}] \right\rangle
=-\frac{1}{\langle \mu , H\rangle } \left\langle (\tau_{H})_{H}, [S_{\mu , j}, T_{\lambda, i}] \right\rangle \\
&=-\frac{1}{\langle \mu , H\rangle } \left\langle -[S_{\mu, j}, (\tau_{H})_{H}], T_{\lambda, i} \right\rangle
=\frac{1}{\langle \mu , H\rangle } \left\langle \langle \mu , (\tau_{H})_{H} \rangle T_{\mu, j}, T_{\lambda, i} \right\rangle\\
&=\frac{\langle \mu , (\tau_{H})_{H} \rangle}{\langle \mu , H\rangle } \delta_{\mu, \lambda }\delta_{j, i}.
\end{align*}
Hence we obtain
\begin{align*}
A_{(\tau_{H})_{H}}T_{\lambda , i}=\frac{\langle \lambda , (\tau_{H})_{H} \rangle}{\langle \lambda , H\rangle }T_{\lambda, i}\quad 
(\lambda \in \Sigma^{+}\setminus \Sigma_{H},\ 1\leq i\leq m(\lambda)).
\end{align*}
Thus, 
\begin{align*}
&B(A_{(\tau_{H})_{H}}T_{\lambda ,i}, T_{\lambda ,i})
=\frac{\langle \lambda , (\tau_{H})_{H} \rangle}{\langle \lambda , H\rangle }B(T_{\lambda ,i}, T_{\lambda ,i})\\
&=\frac{\langle \lambda , (\tau_{H})_{H} \rangle}{\langle \lambda , H\rangle }\frac{[S_{\lambda, i}, T_{\lambda, i}]^{T}}{\langle \lambda , H\rangle}
=\frac{\langle \lambda , (\tau_{H})_{H} \rangle}{\langle \lambda , H\rangle^{2} }\lambda^{T}.
\end{align*}
Therefore, 
we have the consequence.
\end{proof}

The equation~(\ref{biheqn3}) is equivalent to , 
\begin{align*}
0&=\dim (\mathrm{Ad}(K)H)(\tau_{H})_{H}
-\sum_{\lambda \in \Sigma^{+}\setminus \Sigma_{H}}m(\lambda )\frac{\langle \lambda , (\tau_{H})_{H}\rangle}{\langle \lambda , H\rangle^{2}} (\lambda)^{T}\\
&=\dim (\mathrm{Ad}(K)H)(\tau_{H})_{H}
-\sum_{\lambda \in \Sigma^{+}\setminus \Sigma_{H}}m(\lambda )\frac{\langle \lambda , \widetilde{\tau_{H}}- \langle H, \widetilde{\tau_{H}} \rangle H\rangle}{\langle \lambda , H\rangle^{2}} (\lambda)^{T}\\
&=\dim (\mathrm{Ad}(K)H)(\tau_{H})_{H}
-\sum_{\lambda \in \Sigma^{+}\setminus \Sigma_{H}}m(\lambda )
\left( 
\frac{\langle \lambda , \widetilde{\tau_{H}}\rangle}{\langle \lambda , H\rangle^{2}}
-\frac{\langle H, \widetilde{\tau_{H}} \rangle }{\langle \lambda , H\rangle}
\right)(\lambda)^{T}\\
&=\dim (\mathrm{Ad}(K)H)(\tau_{H})_{H}
-\sum_{\lambda \in \Sigma^{+}\setminus \Sigma_{H}}
\left( m(\lambda ) 
\frac{\langle \lambda , \widetilde{\tau_{H}}\rangle}{\langle \lambda , H\rangle^{2}}
(\lambda)^{T}
\right)
+\dim (\mathrm{Ad}(K)H)(\tau_{H})_{H}\\
&=2\dim (\mathrm{Ad}(K)H)(\tau_{H})_{H}
-\sum_{\lambda \in \Sigma^{+}\setminus \Sigma_{H}}
\left( m(\lambda ) 
\frac{\langle \lambda , \widetilde{\tau_{H}}\rangle}{\langle \lambda , H\rangle^{2}}
(\lambda)^{T}
\right).
\end{align*}
Thus, $\mathrm{Ad}(K)H$ in $S$ is biharmonic if and only if 
\begin{align}
(\tau_{2,H})_{H}:=2\dim (\mathrm{Ad}(K)H)(\tau_{H})_{H}
-\sum_{\lambda \in \Sigma^{+}\setminus \Sigma_{H}}
\left( m(\lambda ) 
\frac{\langle \lambda , \widetilde{\tau_{H}}\rangle}{\langle \lambda , H\rangle^{2}}
(\lambda)^{T}
\right)=0
\end{align}
Moreover, 
$\mathrm{Ad}(K)H$ in $S$ is biharmonic if and only if
there exists some constant $c\in \mathbb{R}$, 
$\widetilde{\tau_{2,H}}=cH$ holds. 
\begin{remark}\rm
The vector $(\tau_{2,H})_{H}$ is not necessarily the bitension field of $\mathrm{Ad}(K)H$ in $S$, 
but the condition $(\tau_{2,H})_{H}=0$ is a necessary and sufficient condition for $\mathrm{Ad}(K)H$ to be biharmonic in $S$. Thus, in this paper we use this symbol $(\tau_{2,H})_{H}$.
\end{remark}

Biharmonic orbits can be given by solving Equation~(\ref{biheqn3}) for $H$. 
However, it is difficult to solve this equation in general.
In \cite{HTST}, by using a convex function on $\mathcal{C}^{\Delta}\cap S$ which satisfy $(\mathrm{grad}F)_{H}=(\tau_{H})_{H}$, they show that there exists a  unique $H \in \mathcal{C}^{\Delta}\cap S$ 
such that $(\tau_{H})_{H}=0$ as a critical point of the function.
Even if such a function $f$ on $\mathcal{C}^{\Delta}\cap S$ exists for $(\tau_{2,H})_{H}$, 
it is difficult to decide whether a critical point of $f$ gives a proper biharmonic submanifold or a harmonic submanifold.
Therefore we add some assumptions for $\Sigma $ and $ H $ and discuss the equation $(\tau_{2, H})_{H}=0$.

Hereafter, we assume the root system $\Sigma$ is reducible.
This assumption means that the representation of $K$ on $\mathfrak{m}$ is reducible.
Thus, the orbit $\mathrm{Ad}(K)H$ is a direct product of some R-spaces.
Then the root system $\Sigma$ of $\mathfrak{a}$ is decomposed as $\Sigma = \Sigma_{1}\oplus \Sigma_{2}$,
where $\Sigma_{1}$ and $\Sigma_{2}$ are root system of $\mathrm{Span} (\Sigma_{1})$ and $\mathrm{Span} (\Sigma_{2})$ which satisfy $\mathfrak{a}= \mathrm{Span} (\Sigma_{1}) \oplus \mathrm{Span} (\Sigma_{2})$.
For $\lambda \in \Sigma_{1},\ \mu \in \Sigma_{2}$, $\langle \lambda , \mu \rangle =0$ and $\Sigma=\Sigma_{1}\cup \Sigma_{2}$ hold. 
We set $\Delta_{i}=\Delta \cap \Sigma_{i}$, $\Sigma_{i}^{+}=\Sigma^{+}\cap \Sigma_{i}$ and  
take $H_{i}\in \overline{\mathcal{C}^{\Delta_{i}}}\cap S$ for $i=1, 2$. 
Then we have 
\begin{align*}
\dim \mathrm{Ad}(K)H_{i}=\sum_{\lambda \in \Sigma_{i}^{+}\setminus \Sigma_{H}}m(\lambda).
\end{align*}
For $\theta \in (0, \pi/2)$, we set 
$H=\cos \theta H_{1}+ \sin \theta H_{2}.$ 
Then the tension field of the orbit $\mathrm{Ad}(K)H$ in $S$ is given as 
\begin{align*}
&\widetilde{\tau_{H}}=
-\sum_{\lambda \in \Sigma^{+}\setminus \Sigma_{H}} m(\lambda) \frac{\lambda }{\langle \lambda , H\rangle }\\
=&
-\left( \sum_{\lambda \in \Sigma_{1}^{+}\setminus \Sigma_{H}} m(\lambda) \frac{\lambda }{\langle \lambda , H\rangle }
+\sum_{\mu \in \Sigma_{2}^{+}\setminus \Sigma_{H}} m(\mu) \frac{\mu }{\langle \mu , H\rangle }
\right) \\
=&
-\left( \frac{1}{\cos \theta} \sum_{\lambda \in \Sigma_{1}^{+}\setminus \Sigma_{H}} m(\lambda) \frac{\lambda }{\langle \lambda , H_{1}\rangle }
+\frac{1}{\sin \theta}\sum_{\mu \in \Sigma_{2}^{+}\setminus \Sigma_{H}} m(\mu) \frac{\mu }{\langle \mu , H_{2}\rangle }
\right) \\
=&
\frac{1}{\cos \theta}\widetilde{\tau_{H_{1}}}+\frac{1}{\sin \theta}\widetilde{\tau_{H_{2}}}.
\end{align*}
Further, we suppose that 
\begin{align*}
\widetilde{\tau_{H_{i}}}=-n_{i}H_{i}
\end{align*}
for $i=1, 2$,
where $n_{i}=\dim \mathrm{Ad}(K)H_{i}$. 
This means that the R-space $\mathrm{Ad}(K)H_{i}\subset S$ is minimal in $S$. 
For each nonempty $\Delta_{i}'\subset \Delta_{i}$, there exists a unique 
vector $H_{i}\in \overline{\mathcal{C}^{\Delta_{i}}}\cap S$ which gives 
a minimal R-space (\cite{HTST}).
Since 
\begin{align*}
\widetilde{\tau_{H}}=
-\left(
\frac{n_{1}}{\cos \theta}H_{1}+\frac{n_{2}}{\sin \theta}H_{2}
\right),
\end{align*}
$(\tau_{H})_{H}=0$ if and only if 
\begin{align*}
-\left(
\frac{n_{1}}{\cos \theta}H_{1}+\frac{n_{2}}{\sin \theta}H_{2}
\right)
=-(n_{1}+n_{2})(\cos \theta H_{1}+\sin \theta H_{2}).
\end{align*}
Thus we have 
\begin{align*}
0&=
\left\{ 
\frac{n_{1}}{\cos \theta}-(n_{1}+n_{2})\cos \theta 
\right\} H_{1}
+
\left\{ 
\frac{n_{2}}{\sin \theta}-(n_{1}+n_{2})\sin \theta 
\right\} H_{2}\\
&=
\frac{1}{\cos \theta}
\left\{ 
n_{1}(\sin \theta)^{2}-n_{2}(\cos \theta)^{2} 
\right\} H_{1}
+
\frac{1}{\sin \theta}\left\{ 
n_{2}(\cos \theta)^{2}-n_{1}(\sin \theta)^{2} 
\right\} H_{2}.
\end{align*}
The solution of the above equation is 
\begin{align*}
(\cos \theta)^{2}=\frac{n_{1}}{n_{1}+n_{2}}.
\end{align*} 
Then 
\begin{align*}
(\sin \theta)^{2}=\frac{n_{2}}{n_{1}+n_{2}}
\end{align*}
holds. 

A necessary and sufficient condition for an orbit $\mathrm{Ad}(K)H \subset S$ to be biharmonic 
is there exists $c\in \mathbb{R}$, such that 
$\widetilde{\tau_{2, H}}=cH$. 
To examine the condition $\widetilde{\tau_{2, H}}=cH$, we compute $\widetilde{\tau_{2, H}}$.
Then we have 
\begin{align*}
&\widetilde{\tau_{2,H}}=2\dim (\mathrm{Ad}(K)H)\widetilde{\tau_{H}}
-\sum_{\lambda \in \Sigma^{+}\setminus \Sigma_{H}}
\left( m(\lambda ) 
\frac{\langle \lambda , \widetilde{\tau_{H}}\rangle}{\langle \lambda , H\rangle^{2}}
\lambda
\right)\\
=&2(n_{1}+n_{2})\widetilde{\tau_{H}}
-\sum_{\lambda \in \Sigma_{1}^{+}\setminus \Sigma_{H}}
\left( m(\lambda ) 
\frac{-n_{1}}{(\cos \theta)^{3}} \frac{\langle \lambda , H_{1}\rangle}{\langle \lambda , H_{1}\rangle^{2}}
\lambda
\right)
-\sum_{\mu \in \Sigma_{2}^{+}\setminus \Sigma_{H}}
\left( m(\mu ) 
\frac{-n_{2}}{(\sin \theta)^{3}} \frac{\langle \mu , H_{2}\rangle}{\langle \mu , H_{2}\rangle^{2}}
\mu
\right)\\
=&2(n_{1}+n_{2})\widetilde{\tau_{H}}
-\frac{n_{1}}{(\cos \theta)^{3}}\widetilde{\tau_{H_{1}}}
-\frac{n_{2}}{(\sin \theta)^{3}}\widetilde{\tau_{H_{2}}}\\
=&
-2(n_{1}+n_{2})\left( \frac{n_{1}}{\cos \theta }H_{1}+\frac{n_{2}}{\sin \theta }H_{2} \right)
+\frac{n_{1}^{2}}{(\cos \theta)^{3}}H_{1}
+\frac{n_{2}^{2}}{(\sin \theta)^{3}}H_{2}\\
=&
\frac{1}{\cos \theta}\left\{ -2(n_{1}+n_{2})n_{1}+ \frac{n_{1}^{2}}{(\cos \theta)^{2}}\right\}H_{1}
+
\frac{1}{\sin \theta}\left\{ -2(n_{1}+n_{2})n_{2}+ \frac{n_{2}^{2}}{(\sin \theta)^{2}}\right\}H_{2}.
\end{align*}
Since 
$H=\cos \theta H_{1}+\sin \theta H_{2}$, 
a necessary and sufficient condition for an orbit $\mathrm{Ad}(K)H \subset S$ to be biharmonic 
is there exists $c\in \mathbb{R}$, such that 
\begin{align}
\begin{cases}\displaystyle
\frac{1}{\cos \theta}\left\{ -2(n_{1}+n_{2})n_{1}+ \frac{n_{1}^{2}}{(\cos \theta)^{2}}\right\}=c\cos \theta \\
\displaystyle
\frac{1}{\sin \theta}\left\{ -2(n_{1}+n_{2})n_{2}+ \frac{n_{2}^{2}}{(\sin \theta)^{2}}\right\}=c\sin \theta.
\end{cases}
\end{align}
The above equation holds if and only if 
\begin{align}\label{biheqn4}
\frac{1}{(\cos \theta)^{2}}\left\{ -2(n_{1}+n_{2})n_{1}+ \frac{n_{1}^{2}}{(\cos \theta)^{2}}\right\}
-\frac{1}{(\sin \theta)^{2}}\left\{ -2(n_{1}+n_{2})n_{2}+ \frac{n_{2}^{2}}{(\sin \theta)^{2}}\right\}=0
\end{align}
holds. 
Then, we can calculate the left side of Equation~(\ref{biheqn4}). 
\begin{align*}	
&\frac{1}{(\cos \theta)^{2}}\left\{ -2(n_{1}+n_{2})n_{1}+ \frac{n_{1}^{2}}{(\cos \theta)^{2}}\right\}
-\frac{1}{(\sin \theta)^{2}}\left\{ -2(n_{1}+n_{2})n_{2}+ \frac{n_{2}^{2}}{(\sin \theta)^{2}}\right\}\\
=&
\frac{
(n_{1}(\sin \theta)^{2}-n_{2}(\cos \theta)^{2})^{2}
((\sin \theta)^{2}-(\cos \theta)^{2})
}{(\cos \theta)^{4}(\sin \theta)^{4}}.
\end{align*}
Hence the solutions of Equation~(\ref{biheqn4}) are  
\begin{align}
(\cos \theta)^{2}=\frac{n_{1}}{n_{1}+n_{2}},\ \frac{1}{2}.
\end{align}
Summing up the above arguments, 
we have the following theorem.
\begin{theorem}\label{minimal-times-minimal2}
For $i=1, 2$, let $n_{i}$ be a positive integer and 
let   
$M_{i}\subset S^{n_{i}}(1/\sqrt{2})$ be a minimal R-space.
Then, $M_{1}\times M_{2}\subset S^{n_{1}+n_{2}+1}(1)$
is a proper biharmonic submanifold of $S^{n_{1}+n_{2}+1}(1)$.
When $\dim M_{1}\neq \dim M_{2}$, $M_{1}\times M_{2}\subset S^{n_{1}+n_{2}+1}(1)$ is proper biharmonic.
\end{theorem}

\end{document}